\documentclass[11pt,reqno]{amsart}
\usepackage{amsmath,amssymb,graphicx,amscd}
\renewcommand{\i}{{\mathrm{i}}}
\renewcommand{\d}{{\mathrm{d}}}

\renewcommand{\Re}{{\mathfrak{Re}}}

\newtheorem{thm}{Theorem}[section]

\newtheorem{lem}[thm]{Lemma}
\newtheorem{prop}[thm]{Proposition}
\theoremstyle{definition}
\newtheorem{defn}[thm]{Definition}
\theoremstyle{remark}
\newtheorem{rem}[thm]{Remark}
\theoremstyle{remark}

\numberwithin{equation}{section}

\begin{document}

\title[Barnes G-function and Gamma variables]{The barnes G function and its relations with sums
and products of generalized Gamma convolution variables}
\author{A. Nikeghbali}
 \address{Institut f\"ur Mathematik,
 Universit\"at Z\"urich, Winterthurerstrasse 190,
 CH-8057 Z\"urich,
 Switzerland}
 \email{ashkan.nikeghbali@math.uzh.ch}
\author{Marc Yor}
\address{Laboratoire de Probabilit\'es et Mod\`eles Al\'eatoires,
Universit\'e Pierre et Marie Curie, et CNRS UMR 7599, 175 rue du
Chevaleret F-75013 Paris, France.}


\date{\today}

\begin{abstract}
We give a probabilistic interpretation for the Barnes G-function
which appears in random matrix theory and in analytic number theory
in the important moments conjecture due to Keating-Snaith for the
Riemann zeta function, via the analogy with the characteristic
polynomial of random unitary matrices. We show that the Mellin
transform of the characteristic polynomial of random unitary
matrices and the Barnes G-function are intimately related with
products and sums of gamma, beta and log-gamma variables. In
particular, we show that the law of the modulus of the
characteristic polynomial of random unitary matrices can be
expressed with the help of products of gamma or beta variables, and
that the reciprocal of the Barnes G-function has a
L\'{e}vy-Khintchin type representation.  These results  lead us to
introduce the so called generalized gamma convolution variables.

\end{abstract}
\maketitle
\section{Introduction, motivation and main results}
The Barnes G-function, which was first introduced by Barnes in \cite{barnes} (see also \cite{adamchik}), may be defined via its infinite product
representation:
\begin{equation}\label{barnes}
    G\left(1+z\right)=\left(2\pi\right)^{z/2}\exp\left[-\frac{1}{2}\left[\left(1+\gamma\right)z^{2}+z\right]\right]\prod_{n=1}^{\infty}\left(1+\dfrac{z}{n}\right)^{n}\exp\left[-z+\frac{z^{2}}{2n}\right]
\end{equation}
where $\gamma$ is the Euler constant. 

From (\ref{barnes}), one can
easily deduce the following (useful) development of the logarithm of
$G\left(1+z\right)$ for $|z|<1$:
\begin{equation}\label{logbarnes}
    \log G\left(1+z\right)=\dfrac{z}{2}\left(\log
    \left(2\pi\right)-1\right)-\dfrac{1}{2}\left(1+\gamma\right)z^{2}+\sum_{n=3}^{\infty}\left(-1\right)^{n-1}\zeta\left(n-1\right)\dfrac{z^{n}}{n}
\end{equation}where $\zeta$ denotes the Riemann zeta function.

This Barnes G-function  has recently occurred in the work of Keating and
Snaith \cite{Keating-Snaith} in their celebrated moments conjecture
for the Riemann zeta function. More precisely, they consider the set
of unitary matrices of size $N$, endowed with the Haar probability
measure, and they prove the following results:
\begin{prop}[Keating-Snaith \cite{Keating-Snaith}]\label{Jon-Nina}
If $Z$ denotes the characteristic polynomial of a generic
random unitary matrix, considered at any point of the unit circle
(for example $1$), then the following hold:
\begin{enumerate}
\item For $\lambda$ any complex number satisfying $\Re(\lambda)>-1$:
\begin{equation}\label{momentfini}
    \mathbb{E}_{N}\left[|Z|^{2\lambda}\right]=\prod_{j=1}^{N}\dfrac{\Gamma\left(j\right)\Gamma\left(j+2\lambda\right)}{\left(\Gamma\left(j+\lambda\right)\right)^{2}}
\end{equation}
\item For $\Re(\lambda)>-1$:
\begin{equation}\label{momentinfini}
    \lim_{N\rightarrow\infty}\dfrac{1}{N^{\lambda^{2}}}\mathbb{E}_{N}\left[|Z|^{2\lambda}\right]=\dfrac{\left(G\left(1+\lambda\right)\right)^{2}}{G\left(1+2\lambda\right)}
\end{equation}
\end{enumerate}
\end{prop}
Then, using a random matrix analogy (now called the "Keating-Snaith
philosophy"), they make the following conjecture for the moments of
the Riemann zeta function (see
\cite{Keating-Snaith},\cite{Mezzadri}):
$$\lim_{T\rightarrow\infty}\dfrac{1}{\left(\log T\right)^{\lambda^{2}}}\dfrac{1}{T}\int_{0}^{T}\d t\left|\zeta\left(\dfrac{1}{2}+\i t\right)\right|^{2\lambda}=M\left(\lambda\right)A\left(\lambda\right)$$
where $M\left(\lambda\right)$ is the "random matrix factor"
$$M\left(\lambda\right)=\dfrac{\left(G\left(1+\lambda\right)\right)^{2}}{G\left(1+2\lambda\right)}$$
and $A\left(\lambda\right)$ is the arithmetic factor
$$A\left(\lambda\right)=\prod_{p\in\mathcal{P}}\left[\left(1-\dfrac{1}{p}\right)^{\lambda^{2}}\left(\sum_{m=0}^{\infty}\left(\frac{\Gamma(\lambda+m)}{m!\Gamma(\lambda)}\right)^{2}p^{-m}\right)\right]$$
where, as usual, $\mathcal{P}$ is the set of prime numbers.

Due to the importance of this conjecture, as discussed in several papers in \cite{Mezzadri}, it seems interesting to
obtain  probabilistic interpretations of  the non arithmetic part of the conjecture. More precisely, the aim of this paper is twofold:
\begin{itemize}
\item to give a probabilistic interpretation of the "random matrix
factor" $M\left(\lambda\right)$, and more generally of the Barnes
G-function;
\item to understand better the nature of the limit theorem
(\ref{momentinfini}) and its relations with (generalized)
gamma variables.
\end{itemize}
To this end, we first give a probabilistic translation in Theorem \ref{gamma-barnes} of the
infinite product (\ref{barnes}) in terms of a limiting distribution
involving gamma variables (we note that, concerning the Gamma
function, similar translations have been presented in \cite{L.
Gordon} and \cite{fuchs-letta}). Let us recall that a gamma variable
$\gamma_{a}$ with parameter $a>0$ is distributed as:
\begin{equation}\label{gammadef}
    \mathbb{P}\left(\gamma_{a}\in \d t\right)=\dfrac{t^{a-1}\exp\left(-t\right)\d
    t}{\Gamma\left(a\right)},\;\; t>0
\end{equation} and has Laplace transform
\begin{equation}\label{gammalaplace}
    \mathbb{E}\left[\exp\left(-\lambda
    \gamma_{a}\right)\right]=\dfrac{1}{\left(1+\lambda\right)^{a}},\quad
    \Re(\lambda)>-1
\end{equation}and Mellin transform:
\begin{equation}\label{mellingamma}
\mathbb{E}\left[\left(\gamma_{a}\right)^{s}\right]=\dfrac{\Gamma\left(a+s\right)}{\Gamma\left(a\right)},\quad
    \Re(s)>-a
\end{equation}
\begin{thm}\label{gamma-barnes}
If $\left(\gamma_{n}\right)_{n\geq1}$ are independent gamma random
variables with respective parameters $n$, then for $z$ such that $\Re(z)>-1$:
\begin{equation}\label{premiermellinlimit}
\lim_{N\rightarrow\infty}\dfrac{1}{N^{\frac{z^{2}}{2}}}\mathbb{E}\left[\exp\left(-z\left(\sum_{n=1}^{N}\left(\dfrac{\gamma_{n}}{n}\right)-N\right)\right)\right]=\left(A^{z}\exp\left(\dfrac{z^{2}}{2}\right)G\left(1+z\right)\right)^{-1}
\end{equation}where
\begin{equation}\label{a00}
    A=\sqrt{\dfrac{\mathrm{e}}{2\pi}}.
\end{equation}
\end{thm}
The next theorem gives an identity in law  for the
characteristic polynomial which shall lead to a probabilistic
interpretation of the "random matrix factor":
\begin{thm}\label{probabilistic-explanation}
Let $\Lambda$ denote the generic matrix of $U\left(N\right)$, the
set of unitary matrices, fitted with the   Haar probability measure,
and $Z_{N}\left(\Lambda\right)=\det\left(I-\Lambda\right)$. Then the
following hold:
\begin{enumerate}
\item For $\Re(t)>-1$, we have:
\begin{equation}\label{reKS}
    \mathbb{E}\left[|Z_{N}\left(\Lambda\right)|^{t}\right]=\prod_{j=1}^{N}\dfrac{\Gamma\left(j\right)\Gamma\left(j+t\right)}{\left(\Gamma\left(j+\frac{t}{2}\right)\right)^{2}}
\end{equation}
\item Equivalently, in probabilistic terms:
\begin{equation}\label{KS-gamma}
    \prod_{j=1}^{N}\gamma_{j}\ \stackrel{\mbox{\small law}}{=}\ |Z_{N}\left(\Lambda\right)|\prod_{j=1}^{N}\sqrt{\gamma_{j}\gamma_{j}'}
\end{equation}
\end{enumerate}
where all variables in sight are assumed to be independent, and
$\gamma_{j}$, $\gamma_{j}'$ are gamma random variables with
parameter $j$.
\end{thm}
The Barnes G-function now comes into the picture via the following
limit results:
\begin{thm}\label{Barnes-gammaproduct}
Let $\left(\gamma_{n}\right)_{n\geq1}$ be independent gamma random
variables with respective parameters $n$; then  the following hold:
\begin{enumerate}
\item for any $\lambda$,
with $\Re(\lambda)>-1$, we have:
\begin{equation}\label{mellinlimigamma}
\lim_{N\to\infty}\dfrac{1}{N^{\lambda^{2}/2}}\mathbb{E}\left[\left(\prod_{j=1}^{N}\gamma_{j}\right)^{\lambda}\right]\exp\left(-\lambda\sum_{j=1}^{N}\psi\left(j\right)\right)=\left(A^{\lambda}G\left(1+\lambda\right)\right)^{-1}
\end{equation}where $A=\sqrt{\dfrac{\mathrm{e}}{2\pi}}$ and
$\psi\left(z\right)=\dfrac{\Gamma'\left(z\right)}{\Gamma\left(z\right)}$.
\item consequently, from (\ref{mellinlimigamma}), together with (\ref{KS-gamma}), we recover the limit theorem (\ref{momentinfini}) of Keating and
Snaith: for $\Re(\lambda)>-1$:
\begin{equation*}
    \lim_{N\rightarrow\infty}\dfrac{1}{N^{\lambda^{2}}}\mathbb{E}_{N}\left[|Z|^{2\lambda}\right]=\dfrac{\left(G\left(1+\lambda\right)\right)^{2}}{G\left(1+2\lambda\right)}
\end{equation*}
\end{enumerate}
\end{thm}
We then naturally extend Theorem \ref{gamma-barnes} to the more
general case of sums of the form:
\begin{equation}\label{sum}
    S_{N}=\sum_{n=1}^{N}\left(Y_{n}-\mathbb{E}\left[Y_{n}\right]\right)
\end{equation}where $$Y_{n}=\dfrac{1}{n}\left(y_{1}^{(n)}+\ldots+y_{n}^{(n)}\right),$$ and where $\left(y_{i}^{(n)}\right)_{1\leq i\leq n<\infty}$ are
independent, with the same distribution as a given random variable
$Y$, where $Y$ is a generalized gamma convolution variable (in short GGC), that is an infinitely divisible $\mathbb{R}_{+}$-valued random variable whose L\'{e}vy measure is of the form
\begin{equation}\label{mesurelevyparticulière}
    \nu\left(\d x\right)=\left(\dfrac{\d x}{x}\right)\left(\int\mu\left(\d
    \xi\right)\exp\left(-x\xi\right)\right),
\end{equation}where $\mu\left(\d
    \xi\right)$ is a Radon measure on $\mathbb{R}_{+}$, called the Thorin measure associated to $Y$. We shall
    further assume that $$\int\mu\left(\d
    \xi\right)\dfrac{1}{\xi
    ^{2}}<\infty,$$which, as we shall see is equivalent to the existence
    of a second moment for $Y$.
    
    The GGC variables have been studied by Thorin \cite{thorin} and Bondesson \cite{bondesson}, see, e.g., \cite{jvy} for a recent survey of this topic.
\begin{thm}\label{limitgengam}
Let $Y$ be a GGC variable, and let
$\left(S_{N}\right)$ as in (\ref{sum}). We note
$$\sigma^{2}=\mathbb{E}\left[Y^{2}\right]-\left(\mathbb{E}\left[Y\right]\right)^{2}.$$Then the following limit theorem for $\left(S_{N}\right)$ holds: if
$\lambda>0$,
\begin{equation}\label{generalgammalimit1}
\begin{CD}
\dfrac{1}{N^{\frac{\lambda^{2}\sigma^{2}}{2}}}\mathbb{E}\left[\exp\left(-\lambda
S_{N}\right)\right]@>>N\to\infty> \mathcal{H}\left(\lambda\right),
\end{CD}
\end{equation}where the function $\mathcal{H}\left(\lambda\right)$
is given by:
\begin{equation}\label{Hlambda}
    \mathcal{H}\left(\lambda\right)=\exp\left\{\dfrac{\lambda^{2}\sigma^{2}}{2}\left(1+\gamma\right)+\int_{0}^{\infty}\dfrac{\d y}{y}\Sigma_{\mu}\left(y\right)\left(\exp\left(-\lambda y\right)-1+\lambda
    y-\dfrac{\lambda^{2}y^{2}}{2}\right)\right\},
\end{equation}with
\begin{equation}\label{defdesigmamu}
\Sigma_{\mu}\left(y\right)=\int\mu\left(\d
    \xi\right)\dfrac{1}{\left(2\sinh\left(\dfrac{\xi
    y}{2}\right)\right)^{2}}.
\end{equation}
\end{thm}
Limit results such as (\ref{generalgammalimit1}) are not standard in Probability theory, and we intend to develop a systematic study in a forthcoming paper \cite{JNY}. 
The rest of the paper is organized as follows:
\begin{itemize}
\item in Section 2, we prove Theorems \ref{gamma-barnes},
\ref{probabilistic-explanation} and \ref{Barnes-gammaproduct}. We
also give an interpretation of Theorem \ref{gamma-barnes} in terms
of Bessel processes as well as an important L\'{e}vy-Khintchine type
representation for $1/G\left(1+z\right)$.
\item in Section 3, we compare Theorems \ref{gamma-barnes} and
\ref{Barnes-gammaproduct}, which help us find adequate extensions
for each of them (in particular we shall introduce there generalized
gamma convolution variables and prove Theorem \ref{limitgengam});
\end{itemize}
\section{Proofs of Theorems \ref{gamma-barnes},
\ref{probabilistic-explanation} and \ref{Barnes-gammaproduct} and
additional probabilistic aspects of the Barnes G-function}
\subsection{Proof of Theorem \ref{gamma-barnes} and interpretation in terms of Bessel processes}
\subsubsection{Proof of Theorem \ref{gamma-barnes}}
To prove Theorem  \ref{gamma-barnes}, we simply use the fact that
$$\mathbb{E}\left[\exp\left(-z\dfrac{\gamma_{n}}{n}\right)\right]=\dfrac{1}{\left(1+\dfrac{z}{n}\right)^{n}}.$$Hence,
for $z\geq0$, the quantity
$$\mathbb{E}\left[\exp\left(-z\left(\sum_{n=1}^{N}\left(\dfrac{\gamma_{n}}{n}\right)-N\right)\right)\right]$$equals
$$\dfrac{1}{\prod_{n=1}^{N}\left(1+\dfrac{z}{n}\right)^{n}\exp\left(-z\right)}\equiv\dfrac{1}{D_{N}}.$$We
then write
$$D_{N}\exp\left(\dfrac{z^{2}}{2}\log N\right)=D_{N}\exp\left(\dfrac{z^{2}}{2}\sum_{n=1}^{N}\dfrac{1}{n}\right)\exp\left(\dfrac{z^{2}}{2}\left(\log N-\sum_{n=1}^{N}\dfrac{1}{n}\right)\right)$$
which from (\ref{barnes}) converges, as $N\to\infty$, towards
$$G\left(1+z\right)\left(2\pi\right)^{-z/2}\exp\left(\dfrac{z+z^{2}}{2}\right),$$which
proves formula (\ref{premiermellinlimit}), for $z\geq0$, and thus
for $\Re(z)>-1$ by analytic continuation.
\subsubsection{An interpretation of  Theorem \ref{gamma-barnes} in terms of Bessel processes}
Let $\left(R_{2n}\left(t\right)\right)_{t\geq0}$ denote a
$\mathrm{BES}(2n)$ process, starting from $0$, with dimension $2n$;
we need to consider the sequence
$\left(R_{2n}\right)_{n=1,2,\ldots}$ of such independent processes.
It is well known (see \cite{revuzyor} for example) that:
$$\text{for fixed}\; t>0,\;\;R_{2n}^{2}\left(t\right)\ \stackrel{\mbox{\small law}}{=}\
\left(2t\right)\gamma_{n}.$$Moreover, we have:
$$\text{for fixed}\; t>0,\;\;R_{2n}^{2}\left(t\right)=2\int_{0}^{t}\sqrt{R_{2n}^{2}\left(s\right)}\d
\beta_{s}^{(n)}+2nt$$the stochastic differential equation of
$R_{2n}^{2}$, driven by a Brownian Motion $\beta^{(n)}$. Thus
$$\sum_{n=1}^{N}\dfrac{R_{2n}^{2}\left(t\right)}{2n}\ \stackrel{\mbox{\small law}}{=}\ t\left(\sum_{n=1}^{N}\dfrac{\gamma_{n}}{n}\right).$$We now
write Theorem \ref{gamma-barnes}, for $\Re(z)>-1/t$ as:
\begin{equation}\label{aa}
\lim_{N\rightarrow\infty}\dfrac{1}{N^{\frac{t^{2}z^{2}}{2}}}\mathbb{E}\left[\exp\left(-z\sum_{n=1}^{N}\dfrac{R_{2n}^{2}\left(t\right)-2nt}{2n}\right)\right]=\left(A^{tz}\exp\left(\dfrac{t^{2}z^{2}}{2}\right)G\left(1+tz\right)\right)^{-1}
\end{equation}where
$$A=\sqrt{\dfrac{\mathrm{e}}{2\pi}}.$$We now wish to write the LHS of
(\ref{aa}) in terms of functional of a sum of squared
Orstein-Uhlenbeck processes; indeed, if we write:
$$\sum_{n=1}^{N}\dfrac{R_{2n}^{2}\left(t\right)-2nt}{2n}=2\sum_{n=1}^{N}\int_{0}^{t}\dfrac{1}{2n}\sqrt{R_{2n}^{2}\left(s\right)}\d
\beta_{s}^{(n)},$$the RHS appears as a martingale in $t$, with
increasing process
$$\sum_{n=1}^{N}\int_{0}^{t}\dfrac{1}{n^{2}}R_{2n}^{2}\left(s\right)\d
s,$$and we obtain that (\ref{aa}) may be written as:
$$\lim_{N\rightarrow\infty}\dfrac{1}{N^{\frac{t^{2}z^{2}}{2}}}\mathbb{E}^{(z)}\left[\exp\left(-\dfrac{z}{2}\sum_{n=1}^{N}\int_{0}^{t}\dfrac{1}{n^{2}}R_{2n}^{2}\left(s\right)\d
s\right)\right]=\left(A^{tz}\exp\left(\dfrac{t^{2}z^{2}}{2}\right)G\left(1+tz\right)\right)^{-1}$$where,
under $\mathbb{P}^{(z)}$ the process
$\left(R_{2n}\left(t\right)\right)_{t\geq0}$ satisfies (by
Girsanov's theorem)
\begin{eqnarray*}
  R_{2n}^{2}\left(t\right) &=& 2\int_{0}^{t}\sqrt{R_{2n}^{2}\left(s\right)}\left(\d\widetilde{\beta}_{s}^{(n)}-\dfrac{z}{n}\sqrt{R_{2n}^{2}\left(s\right)}\d s\right)+2nt \\
   &=&
   2\int_{0}^{t}\sqrt{R_{2n}^{2}\left(s\right)}\d\widetilde{\beta}_{s}^{(n)}-\dfrac{2z}{n}\int_{0}^{t}R_{2n}^{2}\left(s\right)\d
   s+2nt
\end{eqnarray*}That is, under $\mathbb{P}^{(z)}$, 
$\left(R_{2n}^{2}\left(t\right)\right)_{t\geq0}$ now appears as the
square of a one dimensional Ornstein-Uhlenbeck process, with
parameter $\left(-\dfrac{z}{n}\right)$.
\subsection{Proof of Theorem \ref{probabilistic-explanation}}
Formula (\ref{reKS}) is due to Keating-Snaith \cite{Keating-Snaith}.
Formula (\ref{KS-gamma}) follows from (\ref{reKS}) once one recalls
formula (\ref{mellingamma}).
\subsection{The characteristic polynomial and beta variables}
One can use the beta-gamma algebra (see, e.g., Chaumont-Yor \cite{chaumont-yor},
p.93-98, and the references therein) and (\ref{KS-gamma}) to represent
(in law) the characteristic polynomial as products of beta
variables. More precisely,
\begin{thm}
With the notations of Theorem \ref{probabilistic-explanation}, we
have:
\begin{equation}\label{aaa}
|Z_{N}|\ \stackrel{\mbox{\small law}}{=}\ 2^{N}\left(\prod_{j=1}^{N}\sqrt{\beta_{\frac{j}{2},\frac{j}{2}}}\right)\left(\prod_{j=2}^{N}\sqrt{\beta_{\frac{j+1}{2},\frac{j-1}{2}}}\right)
\end{equation}
where all the variables in sight are assumed to be independent beta
variables.
\end{thm}
\begin{proof}
To deduce the result stated in the theorem from (\ref{KS-gamma}), we
use the following factorization:
\begin{equation}\label{prov}
    \gamma_{j}\ \stackrel{\mbox{\small law}}{=}\
\sqrt{\gamma_{j}\gamma_{j}'}\xi_{j},
\end{equation}
where on the RHS, we have:
\begin{align*}
\mathrm{for\quad} j=1 & \quad\quad \xi_{1}\ \stackrel{\mbox{\small
law}}{=}\
2\sqrt{\beta_{\frac{1}{2},\frac{1}{2}}} \\
\mathrm{for\quad} j>1 & \quad\quad \xi_{j}\ \stackrel{\mbox{\small
law}}{=}\
2\sqrt{\left(\beta_{\frac{j}{2},\frac{j}{2}}\right)\left(\beta_{\frac{j+1}{2},\frac{j-1}{2}}\right)}.
\end{align*}
Indeed, starting from (\ref{prov}), then multiplying both sides by
$\sqrt{\gamma_{j}\gamma_{j}'}$ and using the beta-gamma algebra, we
obtain that (\ref{prov}) is equivalent to:
$$\gamma_{j}\ \stackrel{\mbox{\small law}}{=}\
2\sqrt{\gamma_{\frac{j}{2}}\gamma_{\frac{j+1}{2}}'}$$which easily
follows from the duplication formula for the gamma function (again, see, e.g., Chaumont-Yor \cite{chaumont-yor},
p.93-98).
\end{proof}
\begin{rem}
In the above Theorem, the factor $2^{N}$ can be explained by the
fact that the characteristic polynomial of a unitary matrix is, in
modulus, smaller than $2^{N}$. Hence the products of beta variables
appearing on the RHS of the formula (\ref{aaa}) measures how the modulus
of the characteristic polynomial deviates from the largest value it
can take.
\end{rem}
\subsection{A L\'{e}vy-Khintchine type representation of $1/G\left(1+z\right)$}
In this subsection, we give a L\'{e}vy-Khintchine representation
type formula for $1/G\left(1+z\right)$ which will be used next to
prove Theorem \ref{Barnes-gammaproduct}.
\begin{prop}\label{Levy-Khintchintyperepres}
For any $z\in\mathbb{C}$, such that $\Re(z)>-1$, one has:
\begin{equation}\label{Levy-Khintchin}
\dfrac{1}{G\left(1+z\right)}=\exp\left\{-\frac{1}{2}\left(\log\left(2\pi\right)-1\right)z+\left(1+\gamma\right)\frac{z^{2}}{2}+\int_{0}^{\infty}\frac{\d
u\;\left(\exp\left(-zu\right)-1+zu-\frac{u^{2}z^{2}}{2}\right)}{u\left(2\sinh\left(u/2\right)\right)^{2}}\right\}
\end{equation}
\end{prop}
Before proving Proposition \ref{Levy-Khintchintyperepres}, a few
remarks are in order.
\begin{rem}
Note that formula (\ref{Levy-Khintchin}) cannot be considered
exactly as a L\'{e}vy-Khintchine representation. Indeed, the integral
featured in (\ref{Levy-Khintchin}) consists in integrating the
function
$\left(\exp\left(-zu\right)-1+zu-\frac{u^{2}z^{2}}{2}\right)$
against the measure $\dfrac{\d
u}{u\left(2\sinh\left(u/2\right)\right)^{2}}$, which is not a
L\'{e}vy measure. Indeed, L\'{e}vy measures integrate
$(u^{2}\wedge1)$, which is not the case here because of the
equivalence: $\begin{CD}
u\left(2\sinh\left(u/2\right)\right)^{2}\sim u^{3} \end{CD}$, when
$u\to0$. Also, due to this singularity, one cannot integrate
$\left(\exp\left(-zu\right)-1+zu\mathbf{1}_{\left(u\leq1\right)}\right)$
with respect to this measure, and one is forced to "bring in" under
the integral sign the companion term $\frac{u^{2}z^{2}}{2}$.
\end{rem}
\begin{proof}
From the consideration of the series
\begin{equation}\label{Lz}
    \mathcal{L}\left(z\right):=\sum_{n=3}^{\infty}\left(-1\right)^{n-1}\zeta\left(n-1\right)\dfrac{z^{n}}{n}
\end{equation}
featured in (\ref{logbarnes}), it seems natural to introduce a
random variable $Q$ taking values in $\mathbb{R}_{+}$, with Mellin
transform:
\begin{equation}\label{Mt}
    \mathbb{E}\left[Q^{s}\right]=\left(\dfrac{\mathbb{E}\left[\mathbf{e}^{s+2}\right]}{2}\right)\left(\dfrac{\zeta\left(s+2\right)}{\zeta\left(2\right)}\right)
\end{equation}where $\mathbf{e}$ denotes a standard exponential
variable. A little more analytically, formula (\ref{Mt}) may be
presented as
\begin{equation}\label{Mtbis}
\mathbb{E}\left[Q^{s}\right]=\dfrac{3}{\pi^{2}}\Gamma\left(s+3\right)\zeta\left(s+2\right).
\end{equation}We first show the existence of the random variable $Q$
by computing its density (for  an example of occurrence of
the random variable $Q$ in the theory of stochastic processes and some relation with the theory of the Riemann Zeta function, see e.g. \cite{bianeyor} and \cite{bianepitmanyor}).

From the definition of $Q$ via its Mellin transform, we get for
$f:\;\mathbb{R}_{+}\to\mathbb{R}_{+}$:
$$\mathbb{E}\left[f\left(Q\right)\right]=\dfrac{3}{\pi^{2}}\sum_{n=1}^{\infty}\left(\dfrac{1}{n^{2}}\int_{0}^{\infty}\d t\;t^{2}\exp\left(-t\right)f\left(\dfrac{t}{n}\right)\right)$$
Hence by some elementary change of variables:
$$\mathbb{P}\left(Q\in\d u\right)=\dfrac{3}{\pi^{2}}u^{2}\left(\sum_{n=1}^{\infty}n\exp\left(-nu\right)\right)\d
u.$$Now, since
$\sum_{n=1}^{\infty}\exp\left(-nu\right)=\dfrac{1}{\exp\left(u\right)-1}$,
we get:\\
$\sum_{n=1}^{\infty}n\exp\left(-nu\right)=\dfrac{1}{\left(2\sinh\left(\frac{u}{2}\right)\right)^{2}}$.
Hence,
\begin{equation}\label{densite}
    \mathbb{P}\left(Q\in\d
u\right)=\dfrac{3}{\pi^{2}}\left(\dfrac{u/2}{\sinh\left(\frac{u}{2}\right)}\right)^{2}\d
u
\end{equation}
Now, we consider the following series development for $|z|<1$:
$$\mathbb{E}\left[\dfrac{1}{Q^{3}}\left(\sum_{n=1}^{\infty}\dfrac{\left(-zQ\right)^{n}}{n!}\right)\right]=\dfrac{3}{\pi^{2}}\sum_{n=3}^{\infty}\left(-1\right)^{n}\zeta\left(n-1\right)\dfrac{z^{n}}{n},$$
hence, in comparison with formula (\ref{Lz}) we obtain:
\begin{eqnarray*}
  \mathcal{L}\left(z\right) &=& -\dfrac{\pi^{2}}{3}\mathbb{E}\left[\dfrac{1}{Q^{3}}\left(\sum_{n=1}^{\infty}\dfrac{\left(-zQ\right)^{n}}{n!}\right)\right] \\
   &=&
   -\dfrac{\pi^{2}}{3}\mathbb{E}\left[\dfrac{1}{Q^{3}}\left(\exp\left(-zQ\right)-1+zQ-\frac{Q^{2}z^{2}}{2}\right)\right].
\end{eqnarray*}
Now, formula (\ref{Levy-Khintchin}), for $|z|<1$, follows from
(\ref{densite}) and the fact that:
\begin{eqnarray*}
  \dfrac{1}{G\left(1+z\right)} &=& \exp\left(\log G\left(1+z\right)\right) \\
   &=& \exp\left\{-\dfrac{z}{2}\left(\log
    \left(2\pi\right)-1\right)+\dfrac{1}{2}\left(1+\gamma\right)z^{2}-\mathcal{L}\left(z\right)\right\}.
\end{eqnarray*}The formula extends by analytic continuation to the
case $\Re(z)>-1$.
\end{proof}

\subsection{Proof of Theorem \ref{Barnes-gammaproduct}}
To prove Theorem \ref{Barnes-gammaproduct}, we shall use the
following lemma:
\begin{lem}\label{lemloggam}
For any $a>0$, the random variable $\log\left(\gamma_{a}\right)$ is
infinitely divisible and its L\'{e}vy-Khintchine representation is
given, for $\Re(\lambda)>-a$, by:
\begin{eqnarray*}
    \mathbb{E}\left[\gamma_{a}^{\lambda}\right] &=&
    \dfrac{\Gamma\left(a+\lambda\right)}{\Gamma\left(a\right)} \\
     &=& \exp\left\{\lambda\psi\left(a\right)+\int_{0}^{\infty}\dfrac{\exp\left(-au\right)\left(\exp\left(-\lambda u\right)-1+\lambda u\right)}{u\left(1-\exp\left(-u\right)\right)}\d u\right\}
\end{eqnarray*}where $$\psi\left(a\right)\equiv\dfrac{\Gamma'\left(a\right)}{\Gamma\left(a\right)}.$$
\end{lem}
\begin{proof}
This is classical: it follows from (\ref{mellingamma}) and some
integral representation of the $\psi$-function, see, e.g. Lebedev
\cite{Lebedev} and Carmona-Petit-Yor \cite{carmona-petit-yor} where
this lemma is also used.
\end{proof}
We are now in a position to prove Theorem \ref{Barnes-gammaproduct}.
We start by proving the first part, i.e. formula
(\ref{mellinlimigamma}). Let us write
$$I_{N}\left(\lambda\right):=\dfrac{1}{N^{\lambda^{2}/2}}\mathbb{E}\left[\left(\prod_{j=1}^{N}\gamma_{j}\right)^{\lambda}\right].$$Thus
with the help of Lemma \ref{lemloggam}, we obtain:
\begin{equation}\label{a1}
N^{\lambda^{2}/2}I_{N}\left(\lambda\right)=\exp\left\{\lambda\sum_{j=1}^{N}\psi\left(j\right)+\int_{0}^{\infty}\dfrac{\sum_{j=1}^{N}\exp\left(-ju\right)}{u\left(1-\exp\left(-u\right)\right)}\left(\exp\left(-\lambda
u\right)-1+\lambda u\right)\d u\right\}
\end{equation}
Note that:
$$\dfrac{1}{u\left(1-\exp\left(-u\right)\right)}\sum_{j=1}^{N}\exp\left(-ju\right)=\dfrac{\exp\left(-u\right)\left(1-\exp\left(-Nu\right)\right)}{u\left(1-\exp\left(-u\right)\right)^{2}}$$
and we may now write:
\begin{equation}\label{a2}
I_{N}\left(\lambda\right)=\exp\left\{\lambda\sum_{j=1}^{N}\psi\left(j\right)+J_{N}\left(\lambda\right)\right\}
\end{equation}where
\begin{equation}\label{a3}
    J_{N}\left(\lambda\right)=\left\{\int_{0}^{\infty}\dfrac{\d u\exp\left(-u\right)}{u\left(1-\exp\left(-u\right)\right)^{2}}\left(1-\exp\left(-Nu\right)\right)\left(\exp\left(-\lambda u\right)-1+\lambda
    u\right)\right\}-\dfrac{\lambda^{2}}{2}\log N
\end{equation}
Next, we shall show that:
$$\begin{CD}
J_{N}\left(\lambda\right)@>>N\to\infty>
J_{\infty}\left(\lambda\right) \end{CD}$$together with some integral
expression for $J_{\infty}\left(\lambda\right)$, from which it will
be easily deduced how $J_{\infty}\left(\lambda\right)$ and
$G\left(1+\lambda\right)$ are related thanks to Proposition
\ref{Levy-Khintchintyperepres}.

We now write (\ref{a3}) in the form:
\begin{equation}\label{a4}
\begin{split}
J_{N}\left(\lambda\right)=\int_{0}^{\infty}\dfrac{\d
u\exp\left(-u\right)}{u\left(1-\exp\left(-u\right)\right)^{2}}\left(1-\exp\left(-Nu\right)\right)\left(\exp\left(-\lambda
u\right)-1+\lambda
    u-\dfrac{\lambda^{2}u^{2}}{2}\right)\\
    \shoveleft{+\dfrac{\lambda^{2}}{2}\left(\int_{0}^{\infty}\dfrac{\d
u
u\exp\left(-u\right)}{\left(1-\exp\left(-u\right)\right)^{2}}\left(1-\exp\left(-Nu\right)\right)-\log
N\right)}
\end{split}
\end{equation}Now letting $N\to\infty$, we obtain:
\begin{equation}\label{a5}
    J_{N}\left(\lambda\right)\to J_{\infty}\left(\lambda\right),
\end{equation}with
\begin{eqnarray*}
  J_{\infty}\left(\lambda\right) &=& \int_{0}^{\infty}\dfrac{\d
u\exp\left(-u\right)}{u\left(1-\exp\left(-u\right)\right)^{2}}\left(\exp\left(-\lambda
u\right)-1+\lambda
    u-\dfrac{\lambda^{2}u^{2}}{2}\right)+\dfrac{\lambda^{2}}{2}C \\
   &\equiv& \widetilde{J}_{\infty}\left(\lambda\right)+\dfrac{\lambda^{2}}{2}C
\end{eqnarray*}
where
$$C=\lim_{N\to\infty}\left(\int_{0}^{\infty}\dfrac{\d
u
u\exp\left(-u\right)}{\left(1-\exp\left(-u\right)\right)^{2}}\left(1-\exp\left(-Nu\right)\right)-\log
N\right),$$a limit which we shall show to exist and identify to be
$1+\gamma(=C)$ with the following lemma:
\begin{lem}\label{bonlem}
We have:
$$\int_{0}^{\infty}\dfrac{u\exp\left(-u\right)}{\left(1-\exp\left(-u\right)\right)^{2}}\left(1-\exp\left(-Nu\right)\right)du=\log
N+\left(1+\gamma\right)+o(1).$$Consequently, we have:
$$C=1+\gamma.$$
\end{lem}
\begin{proof}
\begin{multline*}
\int_{0}^{\infty}\dfrac{u\exp\left(-u\right)}{\left(1-\exp\left(-u\right)\right)^{2}}\left(1-\exp\left(-Nu\right)\right)du=\int_{0}^{\infty}\dfrac{u\exp\left(-u\right)}{\left(1-\exp\left(-u\right)\right)}\left(\sum_{k=0}^{N-1}\exp\left(-ku\right)\right)du\\
=\sum_{k=1}^{N}\int_{0}^{\infty}\dfrac{u}{\left(1-\exp\left(-u\right)\right)}\exp\left(-ku\right)du=\sum_{k=1}^{N}\int_{0}^{\infty}du\;\exp\left(-ku\right)u\sum_{r=0}^{\infty}\exp\left(-ru\right)\\
=\sum_{k=1}^{N}\sum_{r=0}^{\infty}\int_{0}^{\infty}du\;u\exp\left(-\left(r+k\right)u\right)=\sum_{k=1}^{N}\sum_{r=0}^{\infty}\dfrac{1}{\left(r+k\right)^{2}}
\end{multline*}
Now, we write:
\begin{eqnarray*}
  \sum_{k=1}^{N}\sum_{r=0}^{\infty}\dfrac{1}{\left(r+k\right)^{2}} &=& \zeta\left(2\right)+\sum_{k=1}^{N-1}\left(\zeta\left(2\right)-\sum_{s=1}^{k}\dfrac{1}{s^{2}}\right) \\
   &=& N\zeta\left(2\right)-\sum_{k=1}^{N-1}\sum_{s=1}^{k}\dfrac{1}{s^{2}} \\
   &=& N\zeta\left(2\right)-\sum_{s=1}^{N}\dfrac{N-s}{s^{2}} \\
   &=&
   N\sum_{s=N+1}^{\infty}\dfrac{1}{s^{2}}+\dfrac{1}{N}+\sum_{s=1}^{N}\dfrac{1}{s}-\dfrac{1}{N}
\end{eqnarray*}The result now follows easily from the facts:
\begin{eqnarray*}
  \lim_{N\rightarrow\infty}N\sum_{s=N+1}^{\infty}\dfrac{1}{s^{2}} &=& 1 \\
  \sum_{s=1}^{N}\dfrac{1}{s} &=& \log N+\gamma+o(1)
\end{eqnarray*}
\end{proof}
We have thus proved so far that:
\begin{equation}\label{a7}
\lim_{N\to\infty}\dfrac{1}{N^{\lambda^{2}/2}}\mathbb{E}\left[\left(\prod_{j=1}^{N}\gamma_{j}\right)^{\lambda}\right]\exp\left(-\lambda\sum_{j=1}^{N}\psi\left(j\right)\right)=\exp\left(J_{\infty}\left(\lambda\right)\right),
\end{equation}where $$J_{\infty}\left(\lambda\right)=\int_{0}^{\infty}\dfrac{\d
u\exp\left(-u\right)}{u\left(1-\exp\left(-u\right)\right)^{2}}\left(\exp\left(-\lambda
u\right)-1+\lambda
    u-\dfrac{\lambda^{2}u^{2}}{2}\right)+\dfrac{\lambda^{2}}{2}\left(1+\gamma\right).$$We
    can still rewrite $J_{\infty}\left(\lambda\right)$ as:
\begin{equation}\label{a8}
J_{\infty}\left(\lambda\right)=\int_{0}^{\infty}\dfrac{\d
u}{u\left(2\sinh\left(u/2\right)\right)^{2}}\left(\exp\left(-\lambda
u\right)-1+\lambda
    u-\dfrac{\lambda^{2}u^{2}}{2}\right)+\dfrac{\lambda^{2}}{2}\left(1+\gamma\right).
\end{equation}Now comparing (\ref{Levy-Khintchin}) and (\ref{a8}),
we obtain:
\begin{equation}\label{a9}
    \exp\left(J_{\infty}\left(\lambda\right)\right)=\left(A^{\lambda}G\left(1+\lambda\right)\right)^{-1}
\end{equation}Plugging (\ref{a9}) in (\ref{a7}) yields the
first part of Theorem \ref{Barnes-gammaproduct}:
$$\lim_{N\to\infty}\dfrac{1}{N^{\lambda^{2}/2}}\mathbb{E}\left[\left(\prod_{j=1}^{N}\gamma_{j}\right)^{\lambda}\right]\exp\left(-\lambda\sum_{j=1}^{N}\psi\left(j\right)\right)=\left(\sqrt{\dfrac{2\pi}{\mathrm{e}}}\right)^{\lambda}\dfrac{1}{G\left(1+\lambda\right)}$$

To prove the second part of Theorem \ref{Barnes-gammaproduct},
we use formula (\ref{KS-gamma}) together with formula
(\ref{mellinlimigamma}). Formula (\ref{KS-gamma}) yields:
\begin{equation}\label{a10}
\dfrac{1}{N^{2\lambda^{2}}}\mathbb{E}\left[\left(\prod_{j=1}^{N}\gamma_{j}\right)^{2\lambda}\right]=\left(\dfrac{1}{N^{\lambda^{2}}}\mathbb{E}_{N}\left[|Z_{N}|^{2\lambda}\right]\right)\left(\left(\dfrac{1}{N^{\lambda^{2}/2}}\mathbb{E}\left[\left(\prod_{j=1}^{N}\gamma_{j}\right)^{\lambda}\right]\right)^{2}\right)
\end{equation}
Multiplying both sides by
$\exp\left(-2\lambda\sum_{j=1}^{N}\psi\left(j\right)\right)$ and
using (\ref{mellinlimigamma}) we obtain:
$$\lim_{N\rightarrow\infty}\dfrac{1}{N^{\lambda^{2}}}\mathbb{E}_{N}\left[|Z|^{2\lambda}\right]=\dfrac{\left(G\left(1+\lambda\right)\right)^{2}}{G\left(1+2\lambda\right)},$$
which completes the proof of Theorem \ref{Barnes-gammaproduct}.
\section{Generalized Gamma Convolutions}
\subsection{Definition and examples}
We recall the definition of a GGC variable.
\begin{defn}
A random variable $Y$ is called a GGC variable if it
is infinitely divisible with L\'{e}vy measure $\nu$ of the form:
\begin{equation}\label{b1}
    \nu\left(\d x\right)=\dfrac{\d x}{x}\int \mu\left(\d
    \xi\right)\exp\left(-\xi x\right),
\end{equation}where $\mu\left(\d
    \xi\right)$ is a Radon measure on $\mathbb{R}_{+}$, called the Thorin measure of $Y$.
\end{defn}
\begin{rem}
$Y$ is a selfdecomposable random variable because its L\'{e}vy
measure can be written as $\nu\left(\d x\right)=\dfrac{\d
x}{x}h\left(x\right)$ with $h$ a decreasing function (see, e.g.
\cite{Sato}, p.95).
\end{rem}
\begin{rem}
We shall require $Y$ to have finite first and second moments; these
moments can be easily computed with the help of the Thorin measure
$\mu\left(\d
    \xi\right)$:
    \begin{eqnarray*}
      \mathbb{E}\left[Y\right] &=& \mu_{-1}=\int \mu\left(\d
    \xi\right)\dfrac{1}{\xi} \\
      \sigma^{2}=\mathbb{E}\left[Y^{2}\right]-\left(\mathbb{E}\left[Y\right]\right)^{2} &=& \mu_{-2}=\int \mu\left(\d
    \xi\right)\dfrac{1}{\xi^{2}}
    \end{eqnarray*}
\end{rem}
Now we give some examples of GGC variables. Of course,
$\gamma_{a}$ falls into this category with $\mu\left(\d
    \xi\right)=a\delta_{1}\left(\d
    \xi\right)$ where $\delta_{1}\left(\d
    \xi\right)$ is the Dirac measure at $1$.

More generally, the next proposition gives a large set of such
variables:
\begin{prop}
Let $f$ be a nonnegative Borel  function such that
$$\int_{0}^{\infty}\d
u\log\left(1+f\left(u\right)\right)<\infty,$$and let
$\left(\gamma_{u}\right)$ denote the standard gamma process. Then
the variable $Y$ defined as
\begin{equation}\label{b2}
    Y=\int_{0}^{\infty}\d\gamma_{u}f\left(u\right)
\end{equation}is a GGC variable.
\end{prop}
\begin{proof}
It is easily shown, by approximating $f$ by simple
functions  that
\begin{eqnarray*}
  \mathbb{E}\left[\exp\left(-\lambda Y\right)\right] &=& \exp\left(-\int_{0}^{\infty}\d u\int_{0}^{\infty}\dfrac{\d x}{x}\exp\left(-x\right)\left(1-\exp\left(-\lambda f\left(u\right)x\right)\right)\right) \\
   &=& \exp\left(-\int_{0}^{\infty}\dfrac{\d y}{y}\left(\int_{0}^{\infty}\d u\exp\left(-\dfrac{y}{f\left(u\right)}\right)\right)\left(1-\exp\left(-\lambda y\right)\right)\right)
\end{eqnarray*}which yields the result.
\end{proof}
For much more details on GGC variables, see \cite{jvy}.
\subsection{Proof of Theorem \ref{limitgengam}}
We now prove Theorem \ref{limitgengam}, which is a natural extension
for Theorem \ref{gamma-barnes}. Recall that in (\ref{sum}), we have
defined $S_{N}$ as:
$$S_{N}=\sum_{n=1}^{N}\left(Y_{n}-\mathbb{E}\left[Y_{n}\right]\right)$$where $$Y_{n}=\dfrac{1}{n}\left(y_{1}^{(n)}+\ldots+y_{n}^{(n)}\right),$$ and where $\left(y_{i}^{(n)}\right)_{1\leq i\leq n<\infty}$ are
independent, with the same distribution as a given GGC
variable $Y$, which has a second moment. For any $\lambda\geq0$, we
have:
\begin{equation}\label{b3}
\mathbb{E}\left[\exp\left(-\lambda
S_{N}\right)\right]=\prod_{n=1}^{N}\left(\widetilde{\varphi}\left(\dfrac{\lambda}{n}\right)\right)^{n}
\end{equation}where \begin{equation}\label{b4}
\widetilde{\varphi}\left(\lambda\right)=\mathbb{E}\left[\exp\left(-\lambda
\left(Y-\mathbb{E}\left[Y\right]\right)\right)\right].
\end{equation}Now, using the form (\ref{b1}) of the L\'{e}vy-Khintchine representation
for $Y$, we obtain:
\begin{eqnarray}\label{b5}
\prod_{n=1}^{N}\left(\widetilde{\varphi}\left(\dfrac{\lambda}{n}\right)\right)^{n} & = & \exp\left\{-\sum_{n=1}^{N}n\int
\nu\left(\d
x\right)\left(1-\dfrac{\lambda}{n}x-\exp\left(-\dfrac{\lambda}{n}x\right)\right)\right\} \notag \\ 
 & = & \exp\left(-\int\mu\left(\d\xi\right)I_{N}\left(\xi,\lambda\right)\right).
\end{eqnarray}
where:
    \begin{eqnarray*}
      I_{N}\left(\xi,\lambda\right) &=& \sum_{n=1}^{N}n\int_{0}^{\infty}
\dfrac{\d x}{x}\exp\left(-\xi
    x\right)\left(1-\dfrac{\lambda}{n}x-\exp\left(-\dfrac{\lambda}{n}x\right)\right) \\
       &=& \sum_{n=1}^{N}n\int_{0}^{\infty}
\dfrac{\d y}{y}\exp\left(-n\xi
    y\right)\left(1-\lambda y-\exp\left(-\lambda y\right)\right) \\
       &=& \int_{0}^{\infty}
\dfrac{\d y}{y}\left(\sum_{n=1}^{N}n\exp\left(-n\xi
    y\right)\right)\left(1-\lambda y-\exp\left(-\lambda y\right)\right)
    \end{eqnarray*}
Some elementary calculations yield:
$$\sum_{n=1}^{N}n\exp\left(-na\right)=\dfrac{\exp\left(a\right)\left(1-\exp\left(-aN\right)\right)}{\left(\exp\left(a\right)-1\right)^{2}}-\dfrac{N\exp\left(-aN\right)}{\left(\exp\left(a\right)-1\right)}.$$
Consequently, taking $a=\xi y$ in the formula for
$I_{N}\left(\xi,\lambda\right)$, we can write it as:
\begin{equation}\label{b6}
    I_{N}\left(\xi,\lambda\right)=J_{N}\left(\xi,\lambda\right)-R_{N}\left(\xi,\lambda\right)
\end{equation}where \begin{equation}\label{b7}
    J_{N}\left(\xi,\lambda\right)=\int_{0}^{\infty}
\dfrac{\d y}{y}\left(1-\lambda y-\exp\left(-\lambda
y\right)\right)\left[\dfrac{\exp\left(\xi
y\right)\left(1-\exp\left(-\xi yN\right)\right)}{\left(\exp\left(\xi
y\right)-1\right)^{2}}\right]
\end{equation}and
\begin{equation}\label{b8}
R_{N}\left(\xi,\lambda\right)=\int_{0}^{\infty} \dfrac{\d
y}{y}\left(1-\lambda y-\exp\left(-\lambda
y\right)\right)\left[\dfrac{N\exp\left(-\xi
yN\right)}{\left(\exp\left(\xi y\right)-1\right)}\right]
\end{equation}
It is clear that
$$\lim_{N\to\infty}R_{N}\left(\xi,\lambda\right)=0.$$Now let us
study $J_{N}\left(\xi,\lambda\right)$ when $N\to\infty$. Let us
"bring in" the additional term $\dfrac{\lambda^{2}y^{2}}{2}$; more
precisely, we rewrite $J_{N}\left(\xi,\lambda\right)$ as:
\begin{equation}\label{b9}
\begin{split}
J_{N}\left(\xi,\lambda\right)=\int_{0}^{\infty} \dfrac{\d
y}{y}\left(1-\lambda
y+\dfrac{\lambda^{2}y^{2}}{2}-\exp\left(-\lambda
y\right)\right)\left[\dfrac{\exp\left(\xi
y\right)\left(1-\exp\left(-\xi yN\right)\right)}{\left(\exp\left(\xi
y\right)-1\right)^{2}}\right]\\
    \shoveleft{-\dfrac{\lambda^{2}}{2}\int\d y\;y\dfrac{\exp\left(\xi
y\right)\left(1-\exp\left(-\xi yN\right)\right)}{\left(\exp\left(\xi
y\right)-1\right)^{2}}}
\end{split}
\end{equation}
Hence:
\begin{equation}\label{b10}
J_{N}\left(\xi,\lambda\right)=o(1)+\int_{0}^{\infty} \dfrac{\d
y}{y}\dfrac{1}{\left(2\sinh\left(\xi
y/2\right)\right)^{2}}\left(1-\lambda
y+\dfrac{\lambda^{2}y^{2}}{2}-\exp\left(-\lambda
y\right)\right)-\dfrac{\lambda^{2}}{2}K_{N}\left(\xi\right),
\end{equation}where
\begin{eqnarray*}
  K_{N}\left(\xi\right) &=& \int\d y\;y\dfrac{\exp\left(\xi
y\right)\left(1-\exp\left(-\xi yN\right)\right)}{\left(\exp\left(\xi
y\right)-1\right)^{2}} \\
   &=&
   \dfrac{1}{\xi^{2}}\int_{0}^{\infty}\dfrac{u\exp\left(-u\right)}{\left(1-\exp\left(-u\right)\right)^{2}}\left(1-\exp\left(-Nu\right)\right)du.
\end{eqnarray*}
Moreover, from Lemma \ref{bonlem}, we
have:\begin{equation}\label{b11}
K_{N}\left(\xi\right)=\dfrac{1}{\xi^{2}}\left(\log
N+1+\gamma+o(1)\right).
\end{equation}
Now using the above asymptotics, and integrating with respect to
$\mu\left(\d\xi\right)$ (see (\ref{b5})), we obtain:
\begin{multline}\label{b12}
\prod_{n=1}^{N}\left(\widetilde{\varphi}\left(\dfrac{\lambda}{n}\right)\right)^{n}= \\ 
\exp\left\{o(1)+\dfrac{\lambda^{2}}{2}\mu_{-2}\left(\log
N+1+\gamma\right)+\int_{0}^{\infty}\dfrac{\d
y}{y}\Sigma_{\mu}\left(y\right)\left(1-\lambda
y+\dfrac{\lambda^{2}y^{2}}{2}-\exp\left(-\lambda
y\right)\right)\right\}
\end{multline}Hence we have finally proved that if
$\lambda>0$,
\begin{equation}\label{generalgammalimit}
\begin{CD}
\dfrac{1}{N^{\frac{\lambda^{2}\sigma^{2}}{2}}}\mathbb{E}\left[\exp\left(-\lambda
S_{N}\right)\right]@>>N\to\infty> \mathcal{H}\left(\lambda\right),
\end{CD}
\end{equation}where the function $\mathcal{H}\left(\lambda\right)$
is given by:
\begin{equation}\label{Hlambda}
    \mathcal{H}\left(\lambda\right)=\exp\left\{\dfrac{\lambda^{2}\sigma^{2}}{2}\left(1+\gamma\right)+\int_{0}^{\infty}\dfrac{\d y}{y}\Sigma_{\mu}\left(y\right)\left(\exp\left(-\lambda y\right)-1+\lambda
    y-\dfrac{\lambda^{2}y^{2}}{2}\right)\right\},
\end{equation}with
\begin{equation}\label{defdesigmamu}
\Sigma_{\mu}\left(y\right)=\int\mu\left(\d
    \xi\right)\dfrac{1}{\left(2\sinh\left(\dfrac{\xi
    y}{2}\right)\right)^{2}},
\end{equation}which is Theorem \ref{limitgengam}.

\end{document}